\documentclass[conference]{IEEEtran}

\usepackage{graphicx}
\usepackage{epstopdf}
\graphicspath{{Images/}}
\DeclareGraphicsExtensions{.eps,.jpeg,.png}

\usepackage[cmex10]{amsmath}
\usepackage{amsthm}
\usepackage{amsfonts}
\usepackage{cite}%Cite several references at the same time
\usepackage{subfig}
\newtheorem{theorem}{Theorem}
\newtheorem{conjecture}[theorem]{Conjecture}
\newcommand\numberthis{\addtocounter{equation}{1}\tag{\theequation}}

\usepackage{algorithm}
\usepackage{algpseudocode}
\usepackage{pifont}

\begin{document}
%
% paper title
% Titles are generally capitalized except for words such as a, an, and, as,
% at, but, by, for, in, nor, of, on, or, the, to and up, which are usually
% not capitalized unless they are the first or last word of the title.
% Linebreaks \\ can be used within to get better formatting as desired.
% Do not put math or special symbols in the title.
%\title{Optimal 3-D Placement of UAV Cells in Cellular Networks}
\title{Efficient 3-D Placement of an Aerial Base Station in Next Generation Cellular Networks}

% author names and affiliations
% use a multiple column layout for up to three different
% affiliations
%\author{\IEEEauthorblockN{Irem Bor}
%\IEEEauthorblockA{Department of Systems and\\Computer Engineering\\
%Carleton University\\
%Ontario, Ottawa 30332--0250\\
%Email: irembor@sce.carleton.ca}
%\and
%\IEEEauthorblockN{Homer Simpson}
%\IEEEauthorblockA{Twentieth Century Fox\\
%Springfield, USA\\
%Email: homer@thesimpsons.com}
%\and
%\IEEEauthorblockN{James Kirk\\ and Montgomery Scott}
%\IEEEauthorblockA{Starfleet Academy\\
%San Francisco, California 96678--2391\\
%Telephone: (800) 555--1212\\
%Fax: (888) 555--1212}
%}

% conference papers do not typically use \thanks and this command
% is locked out in conference mode. If really needed, such as for
% the acknowledgment of grants, issue a \IEEEoverridecommandlockouts
% after \documentclass

% for over three affiliations, or if they all won't fit within the width
% of the page, use this alternative format:
%
\author{\IEEEauthorblockN{R. Irem Bor Yaliniz,
Amr El-Keyi, and
Halim Yanikomeroglu}
\IEEEauthorblockA{Department of Systems and Computer Engineering, Carleton University, Ottawa, ON, Canada}}

% use for special paper notices
%\IEEEspecialpapernotice{(Invited Paper)}

% make the title area
\maketitle
%\vspace{-10pt}
\makeatletter{\renewcommand*{\@makefnmark}{}
\footnotetext{\hrule \vspace{0.05in}
This work was supported in part by Huawei Canada Co., Ltd., in part by the Ontario Ministry of Economic Development and Innovations ORF-RE (Ontario Research Fund—Research
Excellence) program, and in part by the Natural Sciences and Engineering Research Council of Canada (NSERC) Discovery Grants program.}\makeatother}

% As a general rule, do not put math, special symbols or citations
% in the abstract
\begin{abstract}
Agility and resilience requirements of future cellular networks may not be fully satisfied by terrestrial base stations in cases of unexpected or temporary events. A promising solution is assisting the cellular network via low-altitude unmanned aerial vehicles equipped with base stations, i.e., \textit{drone-cells}. Although drone-cells provide a quick deployment opportunity as aerial base stations, efficient placement becomes one of the key issues. In addition to mobility of the drone-cells in the vertical dimension as well as the horizontal dimension, the differences between the air-to-ground and terrestrial channels cause the placement of the drone-cells to diverge from placement of terrestrial base stations. In this paper, we first highlight the properties of the drone-cell placement problem, and formulate it as a 3-D placement problem with the objective of maximizing the revenue of the network. After some mathematical manipulations, we formulate an equivalent quadratically-constrained mixed integer non-linear optimization problem and propose a computationally efficient numerical solution for this problem. We verify our analytical derivations with numerical simulations and enrich them with discussions which could serve as guidelines for researchers, mobile network operators, and policy makers.
%As the demands of the future cellular networks identified, it became clear that the network will need assistance to provide the expected quality of communication even in the extreme cases, such as natural disasters. One of the key issues is the placement of the UAV. We formulated the a 3-D placement problem for the UAV by considering its locations in the verical and horizontal dimensions.
\end{abstract}

% no keywords

% For peer review papers, you can put extra information on the cover
% page as needed:
% \ifCLASSOPTIONpeerreview
% \begin{center} \bfseries EDICS Category: 3-BBND \end{center}
% \fi
%
% For peerreview papers, this IEEEtran command inserts a page break and
% creates the second title. It will be ignored for other modes.
\IEEEpeerreviewmaketitle
\section{Introduction}
% no \IEEEPARstart
Next generation cellular networks have high reliability and availability demands~\cite{osseiran_scenarios_2014}. Be it a natural disaster, extreme densities of users in an area, or providing connectivity in rural areas, the cellular network needs to meet certain quality of service (QoS) requirements. However, these situations are either unexpected, or temporary. As a result, it is not feasible to invest in an infrastructure that will provide revenue for a relatively short time. A potential solution to these problems is assisting the cellular network via low-altitude unmanned aerial vehicles (UAV) that can serve as aerial base stations with a quick deployment opportunity, i.e. \textit{drone-cells}. However, one of the biggest challenges is to determine the optimal placement of the drone-cell so that the network can benefit the most.

Although there has been significant amount of work on using UAVs in surveillance and reconnaissance networks~\cite{kandeepan_aerial-terrestrial_2014, zhan_wireless_2011}, using UAVs to assist future cellular networks is still at its infancy. In~\cite{rohde_ad_2013} and~\cite{li_drone-assisted_2015}, the positioning of aerial relays is discussed. However, both works have a fixed altitude assumption and place the UAV on a line segment without considering the relation between the coverage area and the altitude of the UAV. Moreover, the effects of urban environment on the performance of communications is not considered. Both issues are addressed in~\cite{al-hourani_optimal_2014}, which provides fundamental results on the optimal altitude of a drone-cell, and the channel model to be utilized in urban environments. Accordingly, the authors of~\cite{drone} investigate the coverage of two drone-cells positioned at a fixed altitude, and interfering with each other over a given area. The effects of interference is further studied in the presence of underlaid device-to-device (D2D) communications in~\cite{mozaffari_unmanned_2015}. In this work, there is no other base station except the drone-cell, which comes with the assumption that the area to be covered is known. %It was shown in~\cite{mozaffari_unmanned_2015} that the maximum average sum-rate can be achieved by adjusting the altitude based on the density of D2D users. %
%This approach can be useful when the region to be covered is known, such as in the case of providing service in rural areas, where the eNB placement is rather sparse and the UAV acts as a base station. 
A similar approach in~\cite{merwaday_uav_2015} shows the improvement in the coverage by assisting the network with drone-cells at a certain altitude, in case of failure of Evolved Node Bs (eNBs). However, these studies neither cover all potential scenarios, nor show the optimal results for all possible selections of altitude, location, and coverage area. For instance, in case of a congested cell, neither the size of the area to be covered, nor the location of that area within the cell is known. These parameters need to be determined according to the target revenue, and the QoS requirements. Moreover, determining the altitude of the UAV is intertwined in this problem because of the characteristics of the channel between the UAV and the terrestrial users, namely the air-to-ground channel.\\
\indent The work presented so far either considers the altitude, which is 1-D, or placement in the horizontal space, which is 2-D. To the best of our knowledge, this paper is the first to propose an efficient 3-D placement algorithm for drone-cells in cellular networks by jointly determining the area to be covered, and the altitude of the drone-cell. We begin by discussing the characteristics of the air-to-ground channel. The discussion on how the placement of a drone-cell is different than the placement of terrestrial base stations leads us to the 3-D placement formulation. Our objective in this formulation is to maximize the revenue of the network, which is proportional to the number of users covered by the drone-cell.\\
\indent Due to the complexity of the channel model, the solution of the 3-D placement problem formulation cannot be directly found. In order to solve this problem, we introduce a new variable relating the altitude of the drone-cell to the radius of its coverage area. Although there is not an analytical expression for the optimal value of this variable, it can be efficiently obtained by using one dimensional bisection search. Afterwards, the 3-D placement problem reduces to a mixed-integer non-linear problem (MINLP), which can be solved by using the interior point optimizer of the MOSEK solver.\\
\indent The rest of the paper is organized as follows. We present the system model, and discuss the channel model in detail in Section~\ref{sec:system_model}. Next, the description of the 3-D placement problem, and the solution method are presented in Section~\ref{sec:place}. Numerical results validating our derivations are presented in Section~\ref{sec:num}. Finally, the paper is concluded in Section~\ref{sec:conc}.

%\hfill mds
%
%\hfill September 17, 2014
%

%%%%%%%%%%%%%%%%%%%%%%%%%%%%%%%%%%%%%%%%%%%%%%%%%%%%%%%%%%%%%%%%%%%%%%%%%%%%%%%%%%%%%%%%%%%%%%%%%%%%%%%%%
%									SYSTEM MODEL
%%%%%%%%%%%%%%%%%%%%%%%%%%%%%%%%%%%%%%%%%%%%%%%%%%%%%%%%%%%%%%%%%%%%%%%%%%%%%%%%%%%%%%%%%%%%%%%%%%%%%%%%%
\section{System Model} \label{sec:system_model}
Consider a macrocell where the location of each user $i$ is known and represented by $(x_i, y_i)$. We assume that for a user to be served, the QoS measured by the received signal-to-noise ratio (SNR) should be above a certain threshold. In case of an extreme event, such as congestion within the cell, or malfunction of the infrastructure, the terrestrial base station may become unable to serve all users. Hence, it will be assisted by a drone-cell with fixed transmission power. We consider a low-altitude quasi-stationary UAV for this purpose, and would like to determine the altitude $h$, and location, $(x_D, y_D)$, providing the maximum revenue. Assuming a fixed QoS for all users, the maximum revenue can be obtained by offloading the maximum amount of users to the drone-cell. The placement of the drone-cell affects both the number of users enclosed by the its coverage region, and the quality of the air-to-ground links. Utilization of air-to-ground links is a characteristic of aerial base stations. There has been several studies on air-to-ground channel models, which we discuss next.
%
%%%%%%%%%%%%%%%%%%%%%%%%%%%%%%%%%%%%%%%%%%%%%%%%%%%%%%%%%%%%%%%%%%%%%%%%%%%%%%%%%%%%%%%%%%%%%%%%%%%%%%%%%
%	Channel Model
%%%%%%%%%%%%%%%%%%%%%%%%%%%%%%%%%%%%%%%%%%%%%%%%%%%%%%%%%%%%%%%%%%%%%%%%%%%%%%%%%%%%%%%%%%%%%%%%%%%%%%%%%
%

The air-to-ground channel differs from the terrestrial channel due to its higher chance of line-of-sight (LoS) connectivity. As a result, Rician~\cite{kandeepan_aerial-terrestrial_2014}, large scale Rayleigh~\cite{zhan_wireless_2011}, and free space fading~\cite{rohde_ad_2013} models are widely utilized in the literature for air-to-ground channels. However, none of them considers the effect of the environment on the occurrence of LoS. One of the most complete models on the effects of building blockage on radio propagation is proposed by ITU in~\cite{itu_prop}. With the help of the results in~\cite{itu_prop}, a channel model for air-to-ground communication in urban environments is presented in~\cite{al-hourani_optimal_2014} and~\cite{al-hourani_modeling_2014}, and adopted here.

The probability of having LoS for user $i$ depends on the altitude of the drone-cell, $h$, and the horizontal distance between the drone-cell and $i^{th}$ user, which is 
$r_i =  \sqrt{(x_D - x_i)^2 + (y_D - y_i)^2 }$ for the $i^{th}$ user located at $(x_i, y_i)$ and the drone-cell at $(x_D, y_D)$. The LoS probability is given by~\cite{al-hourani_optimal_2014}%
\begin{equation}
P(h,r_i) = \frac{1}{1+a\exp\left(-b\left(\arctan\left(\frac{h}{r_i}\right)-a\right)\right)},
\label{eq:plos}
\end{equation}%
where $a$ and $b$ are constant values that depend on the environment. In this setting, the altitude of the user, and the antenna heights of both the users and the drone-cell are neglected. Then the pathloss expression becomes~\cite{al-hourani_optimal_2014}
\begin{multline}
L(h,r_i) = 20\log\left(\frac{4\pi f_c}{c}\right) + 20\log\left(\sqrt{h^2 + r_i^2 }\right) + \\
P(h,r_i)\eta_{\text{LoS}} + (1-P(h,r_i))\eta_{\text{NLoS}},
\label{eq:pl}
\end{multline}
where $f_c$ is the carrier frequency (Hz), $c$ is the speed of light (m/s), $\eta_{LoS}$ and $\eta_{NLoS}$ (in dB) are respectively the losses corresponding to the LoS and non-LoS connections depending on the environment. Equivalently,~\eqref{eq:pl} can be written as%
\begin{equation}
L(h, r_i) = 20\log\left(\sqrt{h^2 + r_i^2 }\right) + AP(h,r_i)+ B,
\label{eq:pl2}
\end{equation} 
where $A$ and $B$ are constants such that $A = \eta_{\text{LoS}} - \eta_{\text{NLoS}}$, and $B = 20\log(\frac{4\pi f_c}{c})+ \eta_{\text{NLoS}}$.
Note that the pathloss model presented here is a function of both $h$ and $r_i$. In other words, the pathloss of the air-to-ground link depends on the altitude in the vertical dimension, and the distance in the horizontal dimension. Thus, we have a 3-D placement problem.

\section{Efficient 3-D Placement of a UAV}\label{sec:place}
Placement of a drone-cell is different than terrestrial cell placement because of the following reasons:
\begin{enumerate}
\item In addition to choosing the drone-cell's location in the horizontal space $(x_D, y_D)$, we need to determine its altitude, $h$, as well.
\item The coverage area of a terrestrial cell is known a priori. However, the coverage area of a drone-cell depends on its altitude, and is unknown before solving the placement problem.
\item The mobility of the drone-cell allows it to move wherever the demand is, rather than terrestrial cells waiting for the demand to come towards them. As a result, the coverage region providing the maximum revenue to the network should be found.
\end{enumerate}
The first item indicates that the placement of the drone-cell is a 3-D problem. In addition, the last two items, which are determining the size of the coverage area, and identifying the location of the coverage area must be considered jointly.  
%In other words, the location of the area to be covered cannot be determined without identifying its size. Moreover, its size cannot be determined without identifying its location and $h$, because the probability of having LoS depends on them.

\begin{figure}[!t]
\centering
\includegraphics[width=0.45\textwidth]{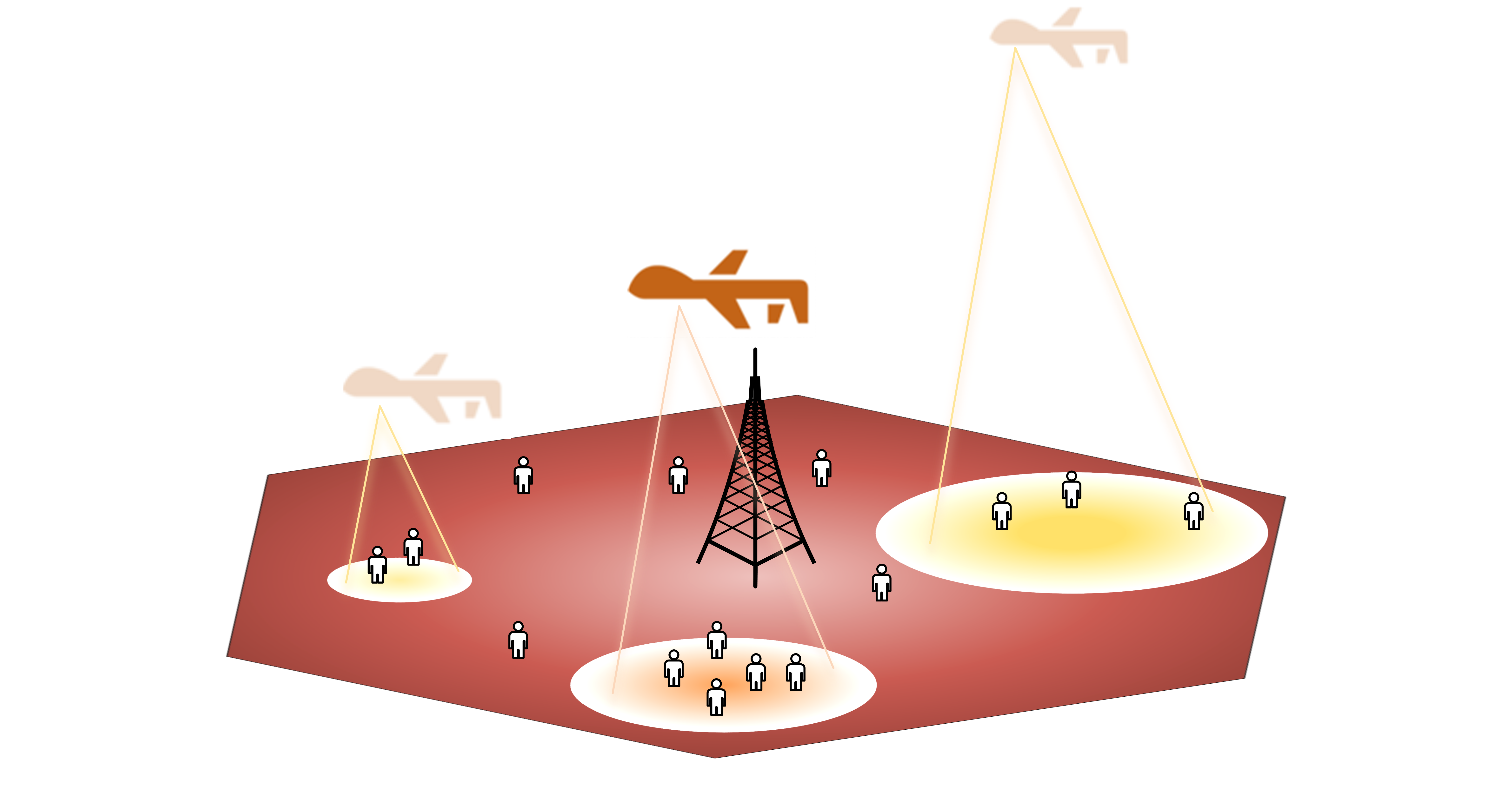}
% where an .eps filename suffix will be assumed under latex,
% and a .pdf suffix will be assumed for pdflatex; or what has been declared
% via \DeclareGraphicsExtensions.
\caption{A possible scenario showing the users that are not covered by the eNB. Three potential placements of a drone-cell is highlighted.}
\label{fig:potential}
\end{figure}
A possible placement problem is shown in Fig.~\ref{fig:potential}. Assume that the macrocell is congested, where only the users that cannot be served by the eNB are shown. Three potential areas to be covered by deploying a drone-cell at different altitudes and locations are highlighted. Note that in each case, as well as the altitude of the drone-cell, the size of the area to be covered is changing, which in turn, affects the number of users served by the drone-cell. In this section, we will formulate and solve the 3-D placement problem efficiently to serve the maximum number of users with the minimum required area.
%
%%%%%%%%%%%%%%%%%%%%%%%%%%%%%%%%%%%%%%%%%%%%%%%%%%%%%%%%%%%%%%%%%%%%%%%%%%%%%%%%%%%%%%%%%%%%%%%%%%%%%%%%%
%	Placement withOUT virtualization
%%%%%%%%%%%%%%%%%%%%%%%%%%%%%%%%%%%%%%%%%%%%%%%%%%%%%%%%%%%%%%%%%%%%%%%%%%%%%%%%%%%%%%%%%%%%%%%%%%%%%%%%%
%

We assume that a user is in the coverage region of the drone-cell if the air-to-ground link satisfies its QoS requirement.  For a given transmission power of the drone-cell, let $\gamma$ represent the pathloss corresponding to the QoS requirement. Hence, user $i$ is served by the drone-cell, if $L(h, r_i) \leq \gamma$. Using~\eqref{eq:pl2}, we can re-write this condition as
\begin{equation}
h^2 + r_i^2 \leq 10^\frac{\gamma - (AP(h,r_i) + B)}{10}.
\label{eq:new_gamma}
\end{equation}
Let $u_i \in \{0,1\}$ denote a binary variable that indicates whether user $i$ is served by the drone-cell, or not. Using the variable $u_i$, which is equal to 1, only if the user $i$ is served by the drone-cell, and equal to 0 otherwise, the following constraint,
\begin{equation}
u_i(h^2 + r_i^2) \leq 10^\frac{\gamma - (AP(h,r_i) + B)}{10},
\label{eq:new_gamma}
\end{equation}
determines whether user $i$ is covered, or not. This constraint can be further manipulated to
\begin{equation}
\begin{aligned}
h^2 + r_i^2\leq 10^\frac{\gamma - (AP(h,r_i) + B)}{10} + M_1(1-u_i),
\label{eq:either}
\end{aligned}
\end{equation}
where $M_1$ is a constant that is slightly larger than the maximum possible value of the distance between a user and the drone-cell. Observe that when $u_{i} = 1$, \eqref{eq:either} is equivalent to \eqref{eq:new_gamma}. If $u_i = 0$, since $M_1$ is large enough, this constraint is released. Now, we can continue by determining the objective function.

Assuming a fixed QoS for all users, the best region to be served by the drone-cell is identified with the maximum number of users covered. By using~\eqref{eq:either}, the placement problem for a set of users, $\mathbb{U}$ that are not covered by the macrocell can be written as
\begin{align*}\label{eq:opt}
&\underset{x_D,y_D,h,\{u_i\}}{\text{maximize}}
\qquad \sum_{i \in\mathbb{U}}u_i\nonumber\\
& \text{subject to}\nonumber\\
& \quad h^2 + r_i^2 \leq 10^\frac{\gamma - (AP(h,r_i) + B)}{10} + M_1(1-u_i),
\ \forall i = 1,...,|\mathbb{U}|,\nonumber\\
& \quad x_{l} \leq x_D \leq x_{u}, \\
& \quad y_{l} \leq y_D \leq y_{u}, \numberthis \\
& \quad h_{l} \leq h \leq h_{u},\nonumber\\
& \quad u_{i} \in \{0,1\},
\qquad \qquad \qquad \qquad \qquad \qquad \quad \forall i = 1,...,|\mathbb{U}|,
\end{align*}
where $|\cdot|$ represents the cardinality of a set, subscripts $(\cdot)_l$ and $(\cdot)_u$ denote respectively the minimum and maximum allowed values for $x_D$, $y_D$, and $h$ of the drone-cell. Note that there are quadratic, exponential and binary terms in this problem, which makes it a MINLP. %
%It is possible to solve MINLP with meta-heuristic methods, but it is hard to claim optimality in that case.
%
We will show that this problem can be solved efficiently by using a combination of the interior-point optimizer of MOSEK solver and bisection search.

Observe that if $P(h,r)$ was a constant, then this optimization problem would be quadratically constrained MINLP. Let us denote the radius of the area to be covered by $R$ and introduce the variable $\alpha$ as
\begin{equation}
\alpha = \frac{h}{R}.
\label{eq:alpha}
\end{equation}
Then, if user $i$ is covered, $R \geq r_{i}$ must be satisfied, i.e., the served user must be located within the coverage region. This conditional expression is similar to~\eqref{eq:new_gamma}, and consequently  is equivalent to
\begin{equation}
R \geq r_i - M_2(1-u_i),
\end{equation}
as in~\eqref{eq:either}, where $M_2$ is a constant value which is slightly larger than the maximum possible value of $R$. Also, the first constraint in~\eqref{eq:opt} becomes
\begin{equation}
R^2 \leq \Gamma(\alpha),
\label{eq:RandGamma}
\end{equation}
where
\begin{equation}
\Gamma(\alpha) = \frac{10^\frac{\gamma - (AP(\alpha) + B)}{10}}{(1+\alpha^2)},
\label{eq:Gamma}
\end{equation}
which enables us to omit the variable $h$ from~\eqref{eq:opt}, since $P(\alpha)$ is
\begin{equation}
P(\alpha) = \frac{1}{1+a\exp\left(-b\left(\arctan\left(\alpha\right)-a\right)\right)}
\end{equation}
by~\eqref{eq:plos}. Thus,~\eqref{eq:opt} becomes

\begin{align*}
\label{eq:fal}
&\underset{x_D,y_D,\{u_i\},R, \alpha}{\text{maximize}}
\quad \sum_{i \in \mathbb{U}}u_i\\
& \text{subject to}
\quad R^2 \leq \Gamma(\alpha)+ M_1(1-u_i),
\quad \forall i = 1,...,|\mathbb{U}|,\\
& \qquad \qquad \quad R \geq r_i - M_2(1-u_i),
\qquad \ \ \forall i = 1,...,|\mathbb{U}|,\\
& \qquad \qquad \quad x_{l} \leq x_D \leq x_{u},\\
& \qquad \qquad \quad y_{l} \leq y_D \leq y_{u},\numberthis \\
& \qquad \qquad \quad R \geq 0,\\
& \qquad \qquad \quad u_i \in \{0,1\},
\qquad \qquad \qquad \quad \  \forall i = 1,...,|\mathbb{U}|.
\end{align*}

Note that if $\alpha$ and $R$ are known, $h$ can be evaluated by using~\eqref{eq:alpha}. Since the variable $\alpha$ appears only in the right-hand-side of the first constraint of~\eqref{eq:fal}, the optimum value of $\alpha$, which maximizes $\Gamma(\alpha)$, maximizes the size of the feasible set of~\eqref{eq:fal}. Next, we numerically show that $\Gamma(\alpha)$ has only one local maxima. Hence, there exists a certain value, $\alpha^*$ that maximizes $\Gamma(\alpha)$.
\begin{conjecture}
For any QoS requirement, $\gamma$, and for any operating frequency, $f_c$, if a local maxima exists in the function $\Gamma(\alpha)$ defined in~\eqref{eq:Gamma}, then it is the only local maxima of the function for $\alpha \in [0, \infty ]$ for the propagation environments whose parameters are listed in Table~\ref{tab:env}.
\label{lemma}
\end{conjecture}
\begin{proof}[Observation]
Observe that $\gamma$ and $B$ in~\eqref{eq:Gamma} only scale the value of~$\Gamma(\alpha)$. Since $B = 20\log(\frac{4\pi f_c}{c})+ \eta_{\text{NLoS}}$, it also follows that the behaviour of $\Gamma(\alpha)$ does not depend on $f_c$. In other words, the maximum point, $\alpha^*$, does not change for different $\gamma$ and $f_c$, but the value of $\Gamma(\alpha^*)$ is scaled. 
%This is depicted in Fig.~\ref{fig:scaling} where $\Gamma(\alpha)$ versus $\alpha$ is plotted for several values of $\gamma$ and $f_c$ with the parameters of urban environment provided in Table~\ref{tab:env}.
%\begin{figure*}[!t]
%\centering
%\subfloat[Scaling effect of $\gamma$ on $\Gamma(\alpha)$]{\includegraphics[width=3.5 in]{gamma_effect.eps}%
%\label{fig:first_case}}
%\hfil
%\subfloat[Scaling effect of $f_c$ on $\Gamma(\alpha)$]{\includegraphics[width=3.5 in]{fc_effect.eps}%
%\label{fig:second_case}}
%\caption{Scaling of $\Gamma(\alpha)$ due to $\gamma$ and $f_c$}
%\label{fig:scaling}
%\end{figure*}

The behaviour of $\Gamma(\alpha)$ is only determined by the environment parameters in $A$ and $P(\alpha)$. By numerically plotting~\eqref{eq:Gamma} in  Fig.~\ref{fig:All_env}, we show that for all environments there exists only one maximum value, which occurs at $\alpha^*$. Moreover, it is observed that the local maximas marked in Fig.~\ref{fig:All_env} are the only maximas for all environments.
\begin{figure}[!t]
\centering
\includegraphics[width=0.45\textwidth]{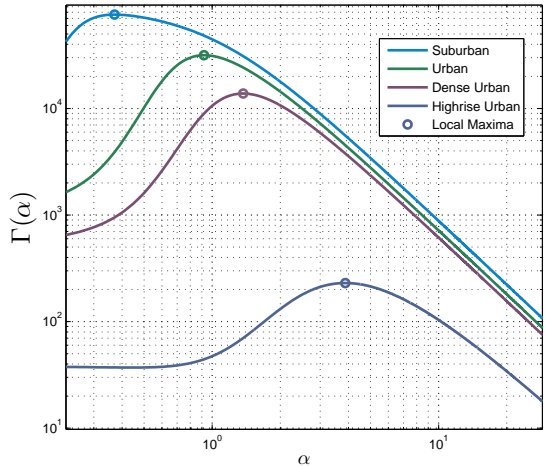}
\caption{$\Gamma(\alpha)$ versus $\alpha$ for various environments.}
\label{fig:All_env}
\end{figure}
\end{proof}
We can use the derivative of $\Gamma(\alpha)$ to find $\alpha^*$, which is the root of $\frac{d\Gamma(\alpha)}{d\alpha}$, that can be calculated as
\begin{equation}
\frac{d\Gamma(\alpha)}{d\alpha} = -\frac{10^{\Lambda}}{\Omega\Delta^2}\left(2\alpha\Delta^2  + AbK \left(\Delta -1\right)\right),
\end{equation}
where $K = 18\frac{\log(10)}{10\log(\mathrm{e})}$, and $\Delta$, $\Lambda$ and $\Omega$ are
%
%Delta
%
\begin{equation}
\Delta = \left( a \exp \left( b \left( a-\frac{180}{\pi}\arctan(\alpha) \right) \right) + 1 \right),\\
\end{equation}
%
%Lambda
%
\begin{equation}
\Lambda = \frac{1}{10} \left(\gamma- B - \frac{A}{\Delta}\right),
\end{equation}
%
%Omega
%
\begin{equation}
\Omega = \left(\alpha^2+1\right)^2.
\end{equation}
%Now, there exists a critical point, $\alpha^*$, such that
%\begin{equation}
%\begin{aligned}
%\frac{d\Gamma(\alpha)}{d\alpha} \geq 0  \iff \alpha < \alpha^* ,\\
%\end{aligned}
%\end{equation}
%meaning that $\Gamma(\alpha)$ is non-decreasing for $\alpha < \alpha^*$, and
%\begin{equation}
%\begin{aligned}
%\frac{d\Gamma(\alpha)}{d\alpha} \leq 0  \iff  \alpha > \alpha^* ,\\
%\end{aligned}
%\end{equation}
%meaning that $\Gamma(\alpha)$ is non-increasing for $\alpha > \alpha^*$, and thus, $\Gamma(\alpha)$ is quasi-convex.
Finally, we proceed with the bisection search to find the root of $\frac{d\Gamma(\alpha)}{d\alpha}$, which is $\alpha^*$. Note that, $\Delta$ yields that the maximum value of $\alpha^*$ can be $\tan(90^\circ)$. Also, the minimum value of $\alpha^*$ is 0, because $\alpha^*$ is a ratio of positive quantities. The bisection search algorithm with a maximum number of iterations, $N_u$, and tolerance, $\epsilon$, can be summarized in Algorithm~\ref{alg:bisect}.
\begin{algorithm}
\caption{Bisection Search Algorithm}
\label{alg:bisect}
\begin{algorithmic}[1]
\State{$N\gets 0$, $\alpha_1 = 0$, $\alpha_2 = \tan(89.9^\circ)$}
\While{$N \leq N_{u}$}
\State{$\alpha_3 \gets \frac{\alpha_1+\alpha_2}{2}$}
\If {$\left(\frac{d\Gamma(\alpha)}{d\alpha}\Bigr|_{\substack{\alpha = \alpha_3}}\right) = 0 $ \textbf{or} $(\alpha_2 - \alpha_1) \leq \epsilon$}
 \State{$\alpha^* = \alpha_3$}
 \State{\textbf{break}}
\EndIf
\State{$N\gets N+1$}
\If{$\text{sign}\left(\frac{d\Gamma(\alpha)}{d\alpha}\Bigr|_{\substack{\alpha = \alpha_3}}\right)  = \text{sign}\left(\frac{d\Gamma(\alpha)}{d\alpha}\Bigr|_{\substack{\alpha = \alpha_1}}\right) $}
\State{$\alpha_1 = \alpha_3$}
\Else
\State {$\alpha_2 = \alpha_3$}
\EndIf
\EndWhile
\end{algorithmic}
\end{algorithm}

After evaluating $\alpha^*$ using Algorithm~\ref{alg:bisect}, the problem given by~\eqref{eq:fal} becomes MINLP, and can be solved to find $x_D$, $y_D$, $\{u_i\}$, and $R$ by interior point optimizer of MOSEK solver. %For a given $\alpha^*$, the problem in~\eqref{eq:fal} is a , which can be efficiently solved using the MOSEK solver.

%%%%%%%%%%%%%%%%%%%%%%%%%%%%%%%%%%%%%%%%%%%%%%%%%%%%%%%%%%%%%%%%%%%%%%%%%%%%%%%%%%%%%%%%%%%%%%%%%%%%%%%%%
%	Placement WITH virtualization
%%%%%%%%%%%%%%%%%%%%%%%%%%%%%%%%%%%%%%%%%%%%%%%%%%%%%%%%%%%%%%%%%%%%%%%%%%%%%%%%%%%%%%%%%%%%%%%%%%%%%%%%%
%Refer to placement_with_virtualization.tex
%
\section{Numerical Results}\label{sec:num}
The numerical values of the parameters of the air-to-ground channels for different environments are calculated based on~\cite{al-hourani_optimal_2014} and~\cite{itu_prop}, and presented in Table~\ref{tab:env}. %
\begin{table}[!t]
\caption{RF Propagation Parameters of different environments}
\label{tab:env}
\centering
\begin{tabular}{|c|c|}
\hline
\textbf{Environment} & \textbf{Parameters ($a$, $b$, $\eta_{\text{LoS}}$, $\eta_{\text{NLoS}})$}\\
\hline
Suburban & (4.88, 0.43, 0.1, 21)\\
\hline
Urban & (9.61, 0.16, 1, 20)\\
\hline
Dense Urban & (12.08, 0.11, 1.6, 23)\\
\hline
High-rise Urban & (27.23, 0.08, 2.3, 34)\\
\hline
\end{tabular}
\end{table}
Also, all simulation parameters are provided in Table~\ref{tab:sim}. It is assumed that the drone-cell have enough capacity to serve all the users in the coverage region. The effect of different environments is shown in Fig.~\ref{fig:env} for 25 users by using $\gamma_2$. After finding $\alpha^*$ by using  Algorithm~\ref{alg:bisect}, the solution of the optimization problem in~\eqref{eq:fal} yields $R$, which determines the size of the circular coverage region, and the location of the drone-cell in 2-D space, as shown with an asterisk in the corresponding color for each environment in Fig.~\ref{fig:env}. Note that not only the size of the region, but also the location of the drone-cell changes. Observe that there are users (some of them are pointed by arrows) right on the edge of the coverage region, which means that the altitude is determined efficiently such that there is no area wasted. As expected, the area covered by the suburban environment has the largest size, due to the reduced blockage compared to other environments. On the other hand, the high-rise urban environment has the worst coverage.
\begin{figure}[!t]
\centering
\includegraphics[width=0.45\textwidth, height=0.45\textwidth]{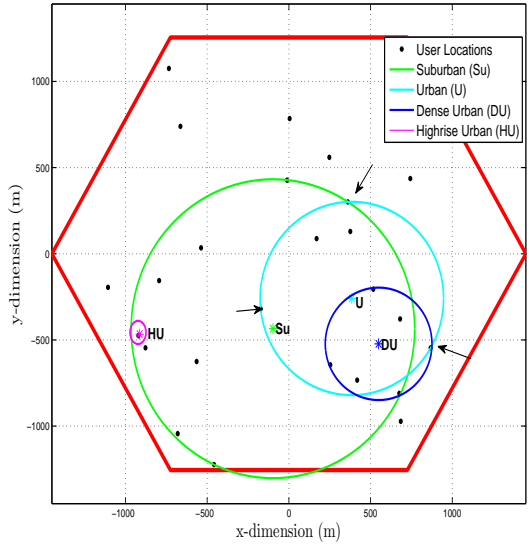}
\caption{Effect of environment on the location and size of the coverage area (circles in the figure) is shown.}
\label{fig:env}
\end{figure}
%We start by a placement example for a simple case in Fig.~\ref{fig:examp}. In this figure, red stars represent the attraction points where the users will more likely to be located. Using $\gamma_2$ and urban environment parameters, $\alpha^*$ is found as $0.92$ by using the bisection search. Then the solution of the optimization problem in~\eqref{eq:fal} yields $R$ as $539.95$ m, and the location of the UAV in 2-D space as shown with blue asterisk in Fig.~\ref{fig:examp}. In addition to the placement of the drone in the center of the most crowded region, observe that there are two users (pointed by arrows) right on the edge of the coverage region. This shows that the altitude is determined such that there is no area wasted. In other words, the placement of UAV is efficient. 
%Fig.~\ref{fig:Gamma} shows $\Gamma(\alpha)$ for suburban and dense urban environments with $\gamma_2$ for the QoS. Note that in both cases, $\Gamma(\alpha)$ is non-decreasing up to $\alpha^*$ (pointed by arrows) and non-increasing afterwards. This point corresponds to the critical point, which can be determined using the bisection search as described in Algorithm~\ref{alg:bisect}.
%
\begin{table}[!t]
\caption{Simulation Parameters}
\label{tab:sim}
\centering
\begin{tabular}{|c|c|}
\hline
\textbf{Parameter} & \textbf{Value}\\
\hline
($x_{l}, x_{u}$) & (-1450, 1450) m\\
\hline
($y_{l}, y_{u}$) & (-1258, 1258) m\\
\hline
($\gamma_1, \gamma_2, \gamma_3$) & (90, 100, 125) dB\\
\hline
$f_c$ & 2.5 GHz\\
\hline
$N_{u}$ & 100\\
\hline
$\epsilon$ & $10^{-5}$\\
\hline
Monte Carlo Runs & 100\\
\hline
\end{tabular}
\end{table}
%
%
%\begin{figure}[!t]
%\centering
%\includegraphics[width=3in,height=3in]{Example_of_placement_25_users.eps}
%\caption{A placement example.}
%\label{fig:examp}
%\end{figure}

To elaborate more on the effect of environment parameters and the performance of the algorithm, we show the average revenue for varying QoS requirements in different environments together with 95$\%$ confidence interval for the revenue in Fig.~\ref{fig:bar}. The results are obtained by using 100 Monte Carlo simulations. In each simulation, 40 users are generated randomly in the cell according to a uniform probability distribution. The results show that the number of served users varies by at most 1 user. Hence, the performance of the proposed method is consistent. The parameter $\gamma_3$ provides the maximum revenue by enabling a coverage area larger than the size of the macrocell, because we are allowing a pathloss of 125 dB for the user to be served. Note that the average revenue for the high-rise urban environment is significantly worse than the other environments for this sparse user distribution. The dramatic drop of revenue can be understood by comparing the parameters of suburban and high-rise urban environments. For instance, $\eta_{\text{NLoS}}$ increases by $13$ dB for the high-rise environment, which alone can reduce the coverage area by more than 100 times. Considering the changes in the other parameters, the significant reduction in the coverage area, and accordingly revenue, is not surprising. However, more users could be covered if the users were in proximity to each other, i.e., clustered. Hence, measuring traffic characteristics in space, such as the amount of clustering as shown in~\cite{cov}, can be of significant importance for determining the efficiency of drone-cell assistance for cellular networks.

\begin{figure}[!t]
\centering
\includegraphics[width=0.45\textwidth]{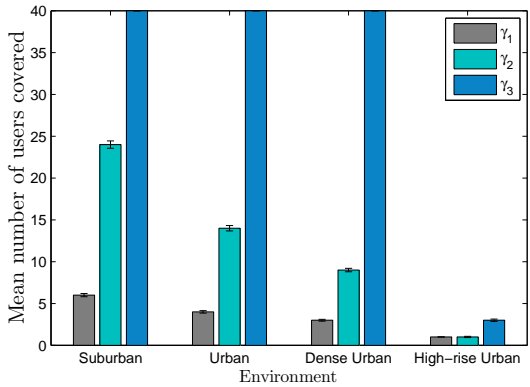}
% where an .eps filename suffix will be assumed under latex,
% and a .pdf suffix will be assumed for pdflatex; or what has been declared
% via \DeclareGraphicsExtensions.
\caption{Mean number of users covered in different environments with 95$\%$ confidence interval. 40 users are distributed uniformly in one macrocell.}
\label{fig:bar}
\end{figure}

\section{Conclusion}\label{sec:conc}

In this paper, we have studied the 3-D placement problem of a drone-cell. First, we discussed the characteristics of the air-to-ground channel, and observed that they can be captured only by considering both the altitude of the drone-cell, and locations of the drone-cell and the users in the horizontal dimension. This yielded a 3-D placement problem with the objective of maximizing the revenue, which is measured by the maximum number of users covered by the drone-cell. We have formulated an equivalent problem which can be solved efficiently to find the location and size of the coverage region, and the altitude of the drone-cell.

Our model can be used for many possible communication scenarios, including failure and congestion. The results presented here can be used by proper authorities to manage and regulate drone-cells assisting cellular networks to meet high demands of the future wireless cellular networks. The effect of interference, and using several drone-cells are interesting future research directions.

% references section
\bibliographystyle{IEEEtran}
\bibliography{Drones_main}

% Generated by IEEEtran.bst, version: 1.13 (2008/09/30)
\begin{thebibliography}{10}
\providecommand{\url}[1]{#1}
\csname url@samestyle\endcsname
\providecommand{\newblock}{\relax}
\providecommand{\bibinfo}[2]{#2}
\providecommand{\BIBentrySTDinterwordspacing}{\spaceskip=0pt\relax}
\providecommand{\BIBentryALTinterwordstretchfactor}{4}
\providecommand{\BIBentryALTinterwordspacing}{\spaceskip=\fontdimen2\font plus
\BIBentryALTinterwordstretchfactor\fontdimen3\font minus
  \fontdimen4\font\relax}
\providecommand{\BIBforeignlanguage}[2]{{%
\expandafter\ifx\csname l@#1\endcsname\relax
\typeout{** WARNING: IEEEtran.bst: No hyphenation pattern has been}%
\typeout{** loaded for the language `#1'. Using the pattern for}%
\typeout{** the default language instead.}%
\else
\language=\csname l@#1\endcsname
\fi
#2}}
\providecommand{\BIBdecl}{\relax}
\BIBdecl

\bibitem{osseiran_scenarios_2014}
A.~Osseiran, F.~Boccardi, V.~Braun, K.~Kusume, P.~Marsch, M.~Maternia,
  O.~Queseth, M.~Schellmann, H.~Schotten, H.~Taoka, H.~Tullberg, M.~Uusitalo,
  B.~Timus, and M.~Fallgren, ``Scenarios for {5G} mobile and wireless
  communications: The vision of the {METIS} project,'' \emph{IEEE Commun.
  Mag.}, vol.~52, no.~5, pp. 26--35, May 2014.

\bibitem{kandeepan_aerial-terrestrial_2014}
S.~Kandeepan, K.~Gomez, L.~Reynaud, and T.~Rasheed, ``Aerial-terrestrial
  communications: Terrestrial cooperation and energy-efficient transmissions to
  aerial base stations,'' \emph{IEEE Trans. on Aerosp. and Electron. Syst.},
  vol.~50, no.~4, pp. 2715--2735, Oct. 2014.

\bibitem{zhan_wireless_2011}
P.~Zhan, K.~Yu, and A.~Swindlehurst, ``Wireless relay communications with
  unmanned aerial vehicles: {Performance} and optimization,'' \emph{IEEE Trans.
  on Aerosp. and Electron. Syst.}, vol.~47, no.~3, pp. 2068--2085, Jul. 2011.

\bibitem{rohde_ad_2013}
\BIBentryALTinterwordspacing
S.~Rohde, M.~Putzke, and C.~Wietfeld, ``Ad-hoc self-healing of {OFDMA} networks
  using {UAV}-based relays,'' \emph{Ad Hoc Networks}, vol.~11, no.~7, pp.
  1893--1906, Sep. 2013. [Online]. Available:
  \url{http://www.sciencedirect.com/science/article/pii/S157087051200131X}
\BIBentrySTDinterwordspacing

\bibitem{li_drone-assisted_2015}
X.~Li, D.~Guo, H.~Yin, and G.~Wei, ``Drone-assisted public safety wireless
  broadband network,'' in \emph{Proc. {IEEE} {Wireless} {Commun.} and {Netw.}
  {Conf.} {Workshops} ({WCNCW})}, Mar. 2015, pp. 323--328.

\bibitem{al-hourani_optimal_2014}
A.~Al-Hourani, S.~Kandeepan, and S.~Lardner, ``Optimal {LAP} altitude for
  maximum coverage,'' \emph{IEEE Wireless Commun. Letters}, vol.~3, no.~6, pp.
  569--572, Dec. 2014.

\bibitem{drone}
\BIBentryALTinterwordspacing
M.~Mozaffari, W.~Saad, M.~Bennis, and M.~Debbah, ``Drone small cells in the
  clouds: Design, deployment and performance analysis,'' \emph{arXiv:1509.01655
  [cs, math]}, Sep. 2015, arXiv: 1509.01655. [Online]. Available:
  \url{http://arxiv.org/abs/1509.01655}
\BIBentrySTDinterwordspacing

\bibitem{mozaffari_unmanned_2015}
\BIBentryALTinterwordspacing
------, ``Unmanned aerial vehicle with underlaid device-to-device
  communications: {Performance} and tradeoffs,'' \emph{arXiv:1509.01187 [cs,
  math]}, Sep. 2015. [Online]. Available: \url{http://arxiv.org/abs/1509.01187}
\BIBentrySTDinterwordspacing

\bibitem{merwaday_uav_2015}
A.~Merwaday and I.~Guvenc, ``{UAV} assisted heterogeneous networks for public
  safety communications,'' in \emph{Proc. {IEEE} {Wireless} {Commun.} and
  {Netw.} {Conf.} {Workshops} ({WCNCW})}, Mar. 2015, pp. 329--334.

\bibitem{itu_prop}
``Propagation data and prediction methods required for the design of
  terrestrial broadband radio access systems operating in a frequency range
  from 3 to 60 {GHz},'' International Telecommunication Union
  Radiocommunication Sector (ITU-R), Recommendation ITU-R P.1410-5, Feb. 2012.

\bibitem{al-hourani_modeling_2014}
A.~Al-Hourani, S.~Kandeepan, and A.~Jamalipour, ``Modeling air-to-ground path
  loss for low altitude platforms in urban environments,'' in \emph{{IEEE}
  {Global} {Commun.} {Conf.} ({GLOBECOM})}, Dec. 2014, pp. 2898--2904.

\bibitem{cov}
M.~Mirahsan, R.~Schoenen, and H.~Yanikomeroglu, ``{HetHetNets}: Heterogeneous
  traffic distribution in heterogeneous wireless cellular networks,''
  \emph{IEEE J. on Sel. Areas in Commun.}, vol.~33, no.~10, October 2015.

\end{thebibliography}

% that's all folks
\end{document}